\documentclass{article}
\usepackage{graphicx} % Required for inserting images
\usepackage{ stmaryrd }
\usepackage{amsmath}
\usepackage{amsthm}
\usepackage{amsfonts}
\usepackage{amssymb}
\usepackage{bbold}
\usepackage[style=alphabetic]{biblatex}
\usepackage{enumitem}
\usepackage{ragged2e}
\newtheorem{theorem}{Theorem}[section]
\newtheorem{lemma}[theorem]{Lemma}
\newtheorem{proposition}[theorem]{Proposition}
\newtheorem{definition}[theorem]{Definition}
\newtheorem{remark}[theorem]{Remark}
\newtheorem{question}[theorem]{Question}
\renewcommand{\geq}{\geqslant}
\renewcommand{\leq}{\leqslant}

\newcommand{\Pro}{\mathbb{P}}
\newcommand{\PR}{\mathbb{P}}
\newcommand{\Z}{\mathbb{Z}}

\newcommand{\N}{\mathbb{N}}
\newcommand{\Aut}{\operatorname{Aut}}
\newcommand{\ee}{\textbf{\textup{e}}}
\newcommand{\xx}{\textbf{\textup{x}}}

\newcommand{\calB}{\mathcal{B}}
\newcommand{\calC}{\mathcal{C}}

\newcommand{\calF}{\mathcal{F}}

\newcommand{\calH}{\mathcal{H}}

\newcommand{\calM}{\mathcal{M}}
\newcommand{\E}{\mathbb{E}}
\newcommand{\Expl}{\operatorname{Expl}}
\usepackage{hyperref}

\title{Percolation on random 2-lifts}
\author{Paul Drouvillé\footnote{paul.drouville@ens-paris-saclay.fr}}
\date{}
\begin{document}
\maketitle
\begin{abstract}
Given a graph $G$, we consider a model for a random cover of $G$ by taking two parallel copies of $G$ and crossing every pair of parallel edges randomly with probability $q$ independently of each other. The resulting graph $G_q$, is a random $2$-lift of $G$ that may not be transitive but still probabilistically exhibit many properties of transitive graphs. Studying percolation in this context can help us test the reliability and robustness of our proofs methods in percolation theory.
Our three main results on this model are the continuity of the critical parameter $p_c(G_q)$,  for $q\in(0,1)$, the strict monotonicity $p_c(G_q)< p_c(G)$ and the exponential decay of the cluster size in the subcritical regime at $q=1/2$.
\end{abstract}

\section{Introduction}
Bernoulli Percolation was introduced in 1957 by Broadbent and Hammersely \cite{BH57} to model the propagation of a fluid through porous materials. While originally focused on Euclidean lattices $\Z^d$, the model was also studied on different classes of graphs. In 1996, Benjamini and Schramm \cite{BS96} initiated a more systematic study of percolation on transitive and quasi-transitive graphs and made connections with geometric group theory. In their article, they study the link between percolation on a graph and percolation on one of its covers. 
If $G$ is a graph and $x$ one its vertex, we write $G_1(x)$ for the subgraph, induces by $x$, its neighbors and the edges that join $x$ to its neighbors. We will say that a graph $H$ is a cover of $G$ if there is a surjective graph homomorphism $\pi$ from $H$ to $G$ such that for any vertex $x$ of $H$ the restriction $\pi:H_1(x)\to G_1(\pi(x))$ is an isomorphism. 
\emph{Throughout this paper, our graphs will always be non-empty simple undirected graphs, connected and locally finite}. We will mostly focus (though not entirely) our study on transitive graphs, that is, graphs such that there is a group that acts transitively on the set of vertices through graph automorphism.

We want to examine a specific model for a random cover of a graph $G$.  We construct our cover by first taking two copies of $G$ in parallel, $G_0$ and $G_1$. For every edge $e=\{x,y\}\in E(G)$, there are two corresponding edges $e_0=\{x_0,y_0\}$ and $e_1=\{x_1,y_1\}$ in $E(G_0)$ and $E(G_1)$. We replace these two edges with $\{x_0,y_1\}$ and $\{x_1,y_0\}$ with probability $q$ and we do nothing with probability $1-q$. See Figure \ref{fig:SwitchDEF} for an illustration. This process is done independently for all edges $e$. We say that the pair is \textbf{switching} if it is in the latter state. 

\begin{figure}[h!]
    \centering
    \includegraphics[scale=0.40]{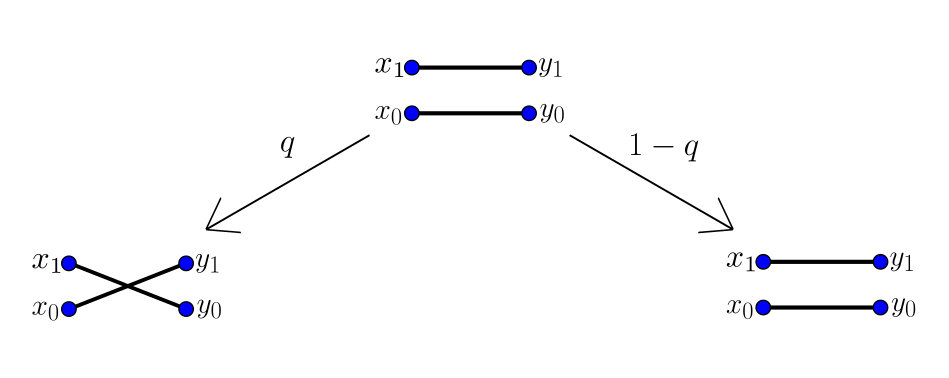}
    \label{fig:SwitchDEF}
    \caption{Pairs of edges can be switched or not, with probability $q$}.
\noindent 
\end{figure}

The resulting random graph is a random 2-lift of the original graph $G$. Constructing a random cover of a graph $G$ has been used in the context of the spectral study of graphs. Recall that a $d$-regular graph $G$ is called a Ramanujan graph if $\lambda_2(G)\leq 2\sqrt{d-1}$ where $\lambda_2(G)$ denotes the second-largest eigenvalue of the Laplacian of $G$. 
Finding the possible values for the number of vertices of a Ramanujan graph, given its degree $d$, is a difficult problem 
For example, one possible solution to create $d$-regular graphs with low second eigenvalue is to start from a $d$-regular graph $G$ that is Ramanujan, and take a random $n$-lift cover of our graph. One can show that the resulting graph is Ramanujan with positive probability (see \cite{MSS14} and \cite{ABG10}).

This article aims to study percolation on this random cover, and we prove in particular three mains theorems:The strict inequality $p_c(G_q)<p_c(G)$ for the critical parameter, the continuity of the function $q\mapsto p_c(q)$ in $(0,1)$ and the exponential decay of the cluster size in the subcritical regime at $q=1/2$.
\begin{theorem}\label{Thm:cont}
    For any transitive infinite graph $G$, the function $q\mapsto p_c(q)$ is continuous on $(0,1)$.
\end{theorem}
In a forthcoming parallel work, Paul Rax independently proves the continuity of $p_c$ at $0$ and $1$, so $q\mapsto p_c(q)$ is fully continuous on $[0,1]$. 

\begin{theorem}\label{Thm:Mono}
    Let $G$ be a connected, transitive infinite graph that is not a tree, and let $q\in(0,1)$. If $p_c(G)<1$ then, we have $$p_c(G_q)<p_c(G),\quad \text{a.s.}$$
\end{theorem}
\begin{remark}
    If $G$ is a tree, then $G_q$ consists essentially of two copies of $G$ in parallel so that $p_c(G)=p_c(G_q)$. Therefore, the hypothesis on the presence of a cycle is necessary.
\end{remark}
\begin{theorem}\label{Thm:Sharp}
    Let $G$ be a transitive infinite graph and $p<p_c(1/2)$. There exist $c,C>0$ such that in $G_{1/2}$, $$\Pro_{p,q}\big[\,|\calC_o|\geq n\big]\leq Ce^{-cn}.$$ 
\end{theorem}

Theorem $\ref{Thm:cont}$ can be put in perspective to the problem of locality of percolation.
Let $(G,o)$ and $(H,o')$ be two rooted graphs. the metric $d$ defined by $$d((H,o'),(G,o))=2^{-\sup\{n|B_n(G,o)\cong B_n(H,o')\}}$$ makes the space of isomorphism classes of rooted graphs a totally disconnected metric space. 
Schramm's Locality Conjecture stated that, restricting our view to transitive graphs where $p_c(G)<1-\varepsilon$ (to avoid cases such as $\Z\times (\Z/N\Z)^d \underset{N\xrightarrow\,\infty}{\longrightarrow} \Z^{d+1}$), then $p_c$ is continuous. This question was solved by Easo and Hutchcroft in \cite{Locality} but counterexamples can be found if the question is extended to random unimodular graph \cite{LocalModul}. In our model, continuously moving the parameter $q$ means continuously changing the law of the graph $G_q$. Therefore, Theorem~\ref{Thm:cont} can be seen as an example of locality in a specific set of random graphs.
The proof in Section \ref{Sec:cont} uses a coupling between percolation and the switching states of edges.

In their seminal article \cite{BS96}, Benjamini and Schramm proved that the critical parameter for a cover graph $H$ is always smaller than the critical parameter for the base graph $G$. They also conjectured that under reasonable hypotheses, the inequality is strict. This question was essentially answered in \cite{MS19} where the authors used augmented percolation to prove the strict monotonicity of critical parameters.
Theorem~\ref{Thm:Mono} is the corresponding result in the context of our random cover. The proof in Section \ref{Sec:Mono} is an adaptation of the arguments of \cite{MS19}. In our model, the hypothesis of ''bounded fibers'' (there exists $D>0$ such that two lifts can be linked in less than $D$ steps) is false. We therefore must accommodate the proof when the former hypothesis is only true ``probabilistically''. 

Theorem $\ref{Thm:Sharp}$ is the famous exponential decay property, which was proved independently by Aizeman and Barsky in \cite{AizeBarsky} in $\Z^d$ and by Menshikov in \cite{Men86} for transitive graphs with slightly-stretched exponential growth. It was then extended to all quasi-transitive graphs in \cite{Antunovi__2007}. More recent proofs can be found in \cite{DCT15,Van23,DRT17}. Just as for locality, counterexamples can be found for unimodular random graphs \cite{LocalModul}. We prove the exponential decay in Section \ref{Sec:Sharp} using the coupling techniques found in \cite{Van23}. 
We need to adapt the definition of an exploration to discover the percolation of edges as well as the structure of the graph itself. We show that natural properties of an exploration for transitive deterministic graphs can be also be true for random covers. The challenge in exploring random graphs is finding what information we need to reveal or not at each step. For example, we may sometimes need to explore an edge without knowing one of its endpoints. We make use of gauge transformations within our random cover that reveal hidden symmetries of our constructions that exist at $q=1/2$. 

We point out that the model we study is not ``monotone'' in the following sense : If $e=\{x,y\}$ is an edge of $G$, then $e'=\{x_0,y_0\}$ and $e''=\{x_0,y_1\}$ are two lifts of $e$  that can be present in $G_q$. However, if $e'$ is open in $G_q$, it is present in $G_q$ and therefore $e''$ is not present and in particular is closed. This means, that any ``positive'' information acquired about percolation comes with a ``negative'' information associated. This property makes studying sharpness on the model harder and more interesting. 

If $x$ is a vertex in a graph $G$, then $B_R(x)$ (resp. $S_R(x)$) will denote the set of vertices at distance at most (resp. exactly) $R$ from $x$. The set $S_{R+1/2}(x)$ will denote the set of edges with one endpoint in $S_{R}(x)$ and the other in $S_{R+1}(x)$. All these notations will be extended to sets of vertices $A$ by taking the union for all the elements in $A$.

To give a formal definition of our random 2-lift,  we consider a graph $G=(V,E)$ and define the modified vertex set $V^\star=V\times\{0,1\}$. For every vertex $v$ in $V$, we write $v_0$ and $v_1$ for the two corresponding lifts in $V^\star$. If $v_{\varepsilon}$ is a lift of $v$, then we call $v_{1-\varepsilon}=T(v_{\varepsilon})$, the ``twin vertex'' of $v_{\varepsilon}$. 
With each edge $e=\{u,v\}$ in, $E$ we associate a Bernoulli variable $\eta_e$ with parameter $q$, the random variables being taken independent of one another. This allows us to define a random edge set: 
$$E(\eta)=\Big\{\{u_0,v_{\eta_e}\}\cup \{u_1,v_{1-\eta_e}\} \,\Big|\,e=\{u,v\}\in E\Big\}.$$
These two set defines our random 2-lift $G_q =(V^\star,E(\eta))$. Let $\pi: G_q\to G$ be the associated covering map. The action of $\pi$ is essentially to forget the distinction between the floors $0$ and $1$. 

We then add another layer of randomness through Bernoulli percolation. Let $(\omega_{e'})_{e'\in E(\eta)}$ be an independent family of Bernoulli variable with parameter $p$. We define $G_{p,q}$ as the random graph of open edges for the Bernoulli percolation on $G_q$. Note that the measures on $\omega$ are conditional on the state of $\eta$, as we need to know which edges exist in our graph to talk about open or closed. Another solution would be to consider edges as an abstract set that may randomly link vertices or not. These considerations do not matter until Section \ref{Sec:Sharp},so we will leave this discussion for now. 
In this paper, variables ``$\eta$'' describe the switching state of pairs of edges, while variables ``$\omega$'' describe a Bernoulli percolation. 

Let $G$ be a transitive graph. Despite the randomness of the graph $G_q$, Bernoulli percolation still exhibits a deterministic phase transition. Let $\Psi(p,q)$ be the probability that there is an infinite cluster in $G_{p,q}$. As in standard Bernoulli percolation for deterministic graphs, Kolmogorov 0-1 law holds and tells us that $\Psi(p,q)\in \{0,1\}$ because the event does not depend on any finite boxes in the graph. The standard coupling also shows that for every $q$ in $[0,1]$, $p\mapsto \Psi(p,q)$ is weakly increasing. 
We can therefore conclude that for every $q$ there is a unique critical parameter $p_c(q)\in[0,1]$ such that:$$\left\{
\begin{array}{ll}
  \Psi(p,q)=1 & \mbox{for }p>p_c(q), \\
  \Psi(p,q)=0 & \mbox{for }p<p_c(q).
\end{array}
\right.$$
Since deg$(G_q)=$ deg$(G)< +\infty$, we have $p_c(q)>0$.
%As we study the properties of the function $q\mapsto p_c(q)$, a question that may arise is the measurability of the function. To prove this property, we can also define $\theta(p,q)$ as the probability that the cluster of a specific root is infinite (the probability does not depend on the choice of the root if the graph is transitive). This event is measurable as the limit of the probability that the root is part of a cluster of size at least $n$. However, $\theta(p,q)$ produces the same phase transition as $\Psi(p,q)$ and therefore $(p_c(q)>a)\Leftrightarrow(\exists b>a,\,\theta(b,q)=0)$. Since, we can rewrite the right term as $(\exists b\in \Q, b>a\, \wedge\, \theta(b,q)=0)$ which is measurable, $q\mapsto p_c(q)$ is a measurable function.

Finally, for a subset $S$ of $V$, we call $\phi_S$, the gauge transformations, which for every vertex $s\in S$, switches $s_0$ with $s_1$, that is $$\phi_S(v_{\varepsilon})=\left\{
\begin{array}{ll}
T(v_{\varepsilon})\text{ if }v\in S,\\
v_{\varepsilon}\text{ if }v\notin S.
\end{array}
\right.$$
We also consider the extension of this mapping to edges also noted $\phi_S$ defined by $\phi_S(\{x,y\})=\{\phi_S(x),\phi_S(y)\}$. The structure of $G_q$ makes $\phi_S$ a graph isomorphism.

\section{Basic properties}
In this section, we give simple properties of the function $q\mapsto p_c(q)$. The first one gives a crude bounds for its range.

\begin{proposition}
 For every $(p,q)$ in $(0,1)^2$, $\theta\left(p^2\right)\leq\theta(p,q)\leq \theta\left(1-(1-p)^2\right)$.
\end{proposition}
\begin{proof}
    With every realization of $G_{p,q}$, we can associate two percolation configurations on $G$ by looking at the image of $G_{p,q}$ by $\pi$. Recall that $\vee$ stands for the maximum operator and $\wedge$ for the minimum operator. Formally, given a family $(\omega^+_{e},\omega^-_{e})_{e\in E}$ of independent Bernoulli variable with parameter $p$ that create a percolation on $G_q$, we define $\omega^{\wedge}_e=\omega^+_{e}\wedge\omega^-_{e}$ and $\omega^{\vee}_e=\omega^+_{e}\vee\omega^-_{e}$. The families $(\omega^{\wedge}_e)_{e\in E}$ and $(\omega^{\vee}_e)_{e\in E}$ are families of independent Bernoulli variables with parameter $p^2$ and $1-(1-p)^2$. The existence of an infinite cluster for $\omega^{\wedge}$ assures the existence of an infinite cluster for $\omega$ which itself assures the existence of an infinite cluster for $\omega^{\vee}$. This means we get the inequality $$\theta(p^2)\leq\theta(p,q)\leq\theta(1-(1-p)^2).$$
\end{proof}
\begin{remark}
The proposition gives us an inequality for $p_c$: $$1-\sqrt{1-p_c(G)}\leq p_c(q)\leq \sqrt{p_c(G)}.$$
\end{remark}
\vspace{3mm}
Recall that a graph is bipartite if there is a partition $V_1,V_2$ of the vertex set such that every edge can be written as $\{v_1,v_2\}$ with $v_1\in V_1$ and $v_2\in V_2$. 
\begin{proposition}\label{Prop:biparti}
    If $G=(V,E)$ is a bipartite graph, then $$\forall \, q\in\,[0,1],\quad p_c(q)=p_c(1-q). $$
\end{proposition}
\begin{proof}
    Let us call $V_1,V_2$ a partition associated with $V$. We consider the effect of the gauge transformation $\phi_{V_1}$. Since every edge has one and only one endpoint in $V_1$, every edge changes its switching state after the application of the morphism. Therefore, if the edge is switching with probability $q$, then its image is switching with probability $1-q$. The action of $\phi_{V_1}$ gives a coupling between $G_q$ and $G_{1-q}$ and therefore between $G_{p,q}$ and $G_{p,1-q}$ such that $\theta(p,q)=\theta(p,1-q)$. Proposition \ref{Prop:biparti} follows.
\end{proof}

In the case $q=1/2$, there is a simple formula for the probability that $G_q$ is connected for a finite graph $G$. If $C$ is a cycle in $G$, we say it is \textbf{switching} in $G_q$ if there is an odd number of edges $e\in C$ such that $\eta_e=1$.
\begin{proposition}
    For $G=(V,E)$ a finite connected graph, the probability that $G_{1/2}$ is not connected is $2^{|V|-|E|-1}$.
\end{proposition}
\begin{proof}
First, note that at $q=1/2$, the probability of switching and not switching are the same. This means that every configuration has the same probability of happening. There are in total $2^{|E|}$ possible configurations, and we need to count how many do not make $G_{1/2}$ connected. In every configuration, we can use the gauge transformations $(\phi_S)_{S \subset V}$, which preserve the connectedness of the configuration. 

The kernel of the action is composed of the subsets $S=\emptyset$ and $S=V$. The factorization theorem tells us that the number of non-connected configurations is $2^{|V|-1}$. We conclude by dividing this last quantity by the total number of configurations, which is $2^{|E|}$. 
\end{proof}

\section{Continuity of $p_c$ through coupling}\label{Sec:cont}
This section aim at proving Theorem~\ref{Thm:Sharp}, which is a consequence of the following theorem. 
\begin{theorem}\label{Holder}
    The function $q\mapsto p_c(q)$ is 1/2-Holderian on every compact interval of $(0,1)$.
\end{theorem}
This theorem is a consequence of the following inequality for the critical parameter.
\begin{lemma}\label{inegalitéstylée}
       Let $b,r$ in $[0,1] $, then for all $a$ in $[(1-r)\,b,(1-r)\,b+r]$, we have $$p_c(b)\geq (1-\sqrt{r})\,p_c(a).$$ 
\end{lemma}

\begin{proof}[Proof of Theorem~\ref{Holder} given Lemma \ref{inegalitéstylée}]
For $I=[\alpha,\beta]\subset(0,1)$, let $q$ and $q'$ be real numbers in $I$ such that $q'<q$. We can write $q'=(1-r)q$ if we set $r=\frac{q-q'}{q}\leq \frac{1}{\alpha}(q-q')$.
Using the inequality from Lemma~\ref{inegalitéstylée} with $a=q'$ and $b=q$ we get $$p_c(q)-p_c(q')\geq  -\frac{1}{\sqrt{\alpha}}\sqrt{q-q'}.$$
On the other side, we can write $q=(1-r)q'+r$ if we set $r=\frac{q-q'}{1-q'}\leq \frac{1}{1-\beta}(q-q')$. Again, using the inequality of Lemma~\ref{inegalitéstylée} with $a = q$ and $ b=q'$, we get $$p_c(q')-p_c(q)\geq - \frac{1}{\sqrt{1-\beta}}\sqrt{q-q'}.$$
We conclude by setting $C=\max\left(\frac{1}{\sqrt{\alpha}},\frac{1}{\sqrt{1-\beta}}\right)$ to get $|p_c(q)-p_c(q')|\leq C{|q-q'|}^{1/2}$ in $[\alpha,\beta$].
\end{proof}

\begin{proof}[Proof of Lemma~\ref{inegalitéstylée}]
Let $0<r<1$ be a real number and $0 \leq q\leq 1-r$. First, let us take a family $(\eta_e)_{e\in E}$ of independent uniform variables in $[0,1]$. By fixing a threshold $a$, we can associate a realization $G_q$ by using the Bernoulli variables $\mathbb{1}_{[0,a]}(\eta_e)$.

Let $(X_e)_{e\in E(G)}$ be a sequence of i.i.d random variables uniform in $\{-1,1\}$ and independent of the variables $(\eta_e)_{e\in E(G)}$. Let $(Y_e)_{e\in E(G)}$ be a sequence of i.i.d Bernoulli variables with parameter $A=2\frac{\sqrt{r}(1-\sqrt{r})}{1-r}=1-\frac{(1-\sqrt{r})^2}{1-r}<1$. The $Y_e$ are independent of every other variables. Let us now define two random variables $\omega^-_{e}$ and $\omega^+_{e}$ as $(\omega^+_{e},\omega^-_{e})=f(\eta_e,X_e,Y_e)$ in the following way:
\begin{itemize}
    \item If $\eta_e\in \: (q,q+r)$, then $\omega^+_{e,}=\omega^-_{e}=0$.
    \item Else, we look at the value of $Y_e$. If it equals $1$, then exactly one of the random variables $\omega^{\varepsilon}_{e}$ is set to $1$ while the other is set to $0$. This choice is made by looking at the sign of $X_e$.
    \item If the last condition fails, we set them both equal to $1$.
\end{itemize}
The reason for the value of the coefficient $A$ is that it is the value that makes $\omega^-_{e}$ and $\omega^+_{e}$ two independent Bernoulli variables with parameter $1-\sqrt{r}$:
\begin{flalign*}
    \Pro(\omega^-_{e}=\omega^+_{e}=1)&=(1-r)\left(1-2\frac{\sqrt{r}(1-\sqrt{r})}{1-r}\right)\\
    &=(1-r)\left(\frac{1-2\sqrt{r}+r}{1-r}\right)= (1-\sqrt{r})^2\\
    &=\Pro(\omega^+_{e}=1)\Pro(\omega^-_{e}=1).
\end{flalign*}
Note that by construction, all $\omega^{\varepsilon}_e$ are independent when $e$ varies. 
let us then define a random switching variable $\hat{\eta}_e$ independently for every edge in $G$. We set $\hat{\eta}_e=g(\eta_e,Z_e)$ where $(Z_e)_{e\in E(G)}$ is a sequence of i.i.d Bernoulli variables with parameter $\frac{q}{1-r}$ independent of everything else so far.
We construct $g$ as follows:
\begin{itemize}
    \item If $\eta_e\leq q$, then we set $\hat{\eta}_e=1$,
    \item If $\eta_e> q+r$, then we set $\hat{\eta}_e=0$,
    \item If $q<\eta_e\leq q+r$, then we set $\hat{\eta}_e=Z_e$ .
\end{itemize}
This ensures all $\hat{\eta}_e$ are independent Bernoulli variables with parameter $\frac{q}{1-r}$. 
Both families of variables are mutually independent, despite both``depending'' on $(\eta_e)_{e\in E}$. 
For example, we can calculate 

\begin{flalign*}
    \Pro(\hat{\eta}_e =1,\,\omega^-_{e}=\omega^+_{e}=1)
&=(1-q-r)\left(1-2\frac{\sqrt{r}(1-\sqrt{r})}{1-r}\right)\\ 
&=\left(1-q-r\right)\frac{\left(1-2\sqrt{r}+r\right)}{1-r}\\
&=\left(1-\frac{q}{1-r}\right)(1-\sqrt{r})^2\\
&=\Pro(\hat{\eta}_e =1)\Pro(\omega^+_{e}=1)\Pro(\omega^-_{e}=1).\\
\end{flalign*}
\begin{flalign*}
\Pro(\hat{\eta}_e =1,\,\omega^-_{e}=1)&=\Pro(\hat{\eta}_e =1,\,\omega^+_{e}=1)\\
&=q\left(1-\frac{\sqrt{r}(1-\sqrt{r})}{1-r}\right)\\
&=\frac{q}{1-r}(1-\sqrt{r})\\
&=\Pro(\hat{\eta}_e =1)\Pro(\omega^+_{e}=1)=\Pro(\hat{\eta}_e =1)\Pro(\omega^-_{e}=1)
\end{flalign*}. \\
These are all the calculations needed to prove independence for the Bernoulli variables.
Then, $(\omega,\hat{\eta})$ allows us to define a random graph $\hat{G}=(V^\star,E(\hat{\eta}))$ with percolation $\omega$ that has the same law as $G_{1-\sqrt{r},\frac{q}{1-r}}$.
If $a$ is real number in $[q,q+r]$, we also define $\overline{\eta}_e=\mathbb{1}_{[0,a]}(\eta_e)$ for every edge $e$ in $E$. This allows us to construct a random graph $\overline{G}=(V^\star,E(\overline{\eta}))$ with the same law as $G_a=G_{1,a}$. The coupling between $\overline{G}$ and $\hat{G}$ is such that they both agree on the  switching state of the edges which are open in $\hat{G}$. Also, every edge open in  $\hat{G}$ is open in $\overline{G}$ and is in the same switching state. The reason for this property is that $\hat{G}$ and $\overline{G}$ only disagree on the switching state of pairs of edges where $\eta_e\in [q,q+r]$ but these edges are by definition closed in $\hat{G}$. 

Independently of everything done so far and independently for all $e$, we add a random percolation of parameter $p$ in both these graphs to create of coupling between $G_{p(1-\sqrt{r}),\frac{q}{1-r}}$ and $G_{p,a}$. More formally, for every edge $e$ in $\overline{G}$, we close it with probability $1-p$ independently of everything done so far. If $e$ is also present in $\hat{G}$ and $e$ was closed in $\overline{G}$, we also close it in $\hat{G}$. This procedure create a new coupled percolation between $\hat{G}$ and $\overline{G}$. 

The result retains the following coupling property:If the cluster of the origin is infinite in $G_{p(1-\sqrt{r}),\frac{q}{1-r}}$, it is also the case for $G_{p,a}$. Note that the addition of a new percolation to both graphs is unambiguous, since every pair of edges that are not  fully closed are in the same switching state. 
We can therefore derive the inequality:
$$\theta\left(p\,(1-\sqrt{r}),\frac{q}{1-r}\right)\leq \theta(p,a).$$ 
meaning that 
\begin{flalign*}
    \theta\left(p\,(1-\sqrt{r}),\frac{q}{1-r}\right)>0 &\Longrightarrow\theta(p,a)>0,\\
    p\,(1-\sqrt{r})>p_c(\frac{q}{1-r}) &\Longrightarrow p>p_c(a),\\
p_c\left(\frac{q}{1-r}\right)&\geq (1-\sqrt{r})\,p_c(a)\quad \text{for all $a$ in $(q,q+r)$}.\end{flalign*}
 We conclude by setting $b=\frac{q}{1-r}\in [0,1]$.
 \end{proof}

\section{Strict monotonicity}\label{Sec:Mono}
Note that while strict monotonicity was proved in \cite{MS19} for covering graphs with strong hypothesis on the type of cover, weak monotonicity for general cover graphs was proved in \cite{BS96} In their article, Benjamini and Schramm show that percolation in the cover graph is more probable than in the base graph. In particular, this implies weak monotonicity of the critical parameter. 
We can directly use their result on every realization of our graph which shows that $p_c(G_q(\eta))\leq p_c(G)$. and we thus have,  $p_c(q)\leq p_c(G)$. Therefore, only the strict monotonicity is a difficult problem. In this section, we adapt the arguments of \cite{MS19} to prove Theorem~\ref{Thm:Mono}. The main difference lies in the fact that we do not control the geometry of our graph uniformly but only probabilistically. Making a precise statement out of the last sentence is the subject of Lemmas \ref{lem:BouclePartition}, \ref{ChoiceR} and \ref{lem:distrib-mag}.

\subsection{Preliminary lemmas}

Similarly to \cite{MS19}, we can dissect the proof of monotonicity into two main inequalities. To compare percolation on $G_q$ and $G$, we use the notion of \textbf{enhanced percolation} --- see \cite{AG91,BBR}. Our enhanced percolation is defined in the following way:
If during our exploration in $G$ the $r$-ball of a vertex $x$ is fully open (meaning every edge with both sides lying in the $r$-ball is open) then, with probability $s$, we can consider every vertex at distance exactly $r+1$ to be part of our connected component. The value $r$ et $s$ will be chosen later. 

Formally, $(\alpha_v)_{v\in V(G)}$ is of family of independent Bernoulli variables with parameter $s$ and independent of $\omega$. Then, given a configuration $(\omega,\alpha)\in \{0,1\}^{E}\times \{0,1\}^V$, the cluster of a vertex $o$ in enhanced percolation is constructed recursively starting with $\calC_o=\{o\}$ by alternating steps. 
At odd steps, we define $C_{2N+1}$ to be the union of all the $\omega$-clusters (that is, the clusters for usual percolation) of the vertices in $C_{2N}$. 
At even steps, $C_{2N+2}$ is the union of $C_{2N+1}$ and the set of vertices $v$ such that there is a $u\in C_{2N+1}$ that satisfies the following conditions :
\begin{enumerate}
    \item $d(u,v)=r+1$,
    \item $B_r(u)\subset C_{2N+1}$,
    \item $\omega_e=1$ for all edges with endpoints in $B_r(u)$,
    \item $\alpha_u=1$. 
\end{enumerate}
After an infinite number of steps, the cluster of $o$ is defined as the infinite union of the $C_n$. 

Note that the power given by augmentation is not symmetric. If the augmentation allows one to explore the vertex $y$ starting from $x$, it does not mean the vertex $x$ is visited in the exploration started at $y$. For a fixed parameter $s$, there has to exist a fixed critical parameter $p_c(G,s)$ such that $o$ is connected to infinity with positive probability if $p>p_c(G,s)$ and is almost surely not connected to infinity if $p<p_c(G,s)$. If $H$ is a (potentially random) graph, we will write $C^p_H(x)$ for the cluster of the vertex $x$ in $H$ for usual percolation and $\calC^{p,s}_H(x)$ for the cluster of $x$ in augmented percolation.
Therefore, to prove strict monotonicity it suffices to choose a correct parameter $s$ and prove these two inequalities

$$p_c(G_q)\overset{(1)}{\leq}p_c(G,s)\overset{(2)}{<} p_c(G),$$
which can be summed up in these two propositions.   

\begin{proposition}\label{Monoto1}
    Let $q$ in $(0,1)$, and $G$ be an infinite transitive graph that is not a tree. Let $o$ be a distinguished vertex of $G$ and $o'$ a lift of $o$ in $G_q$. Then there is a choice of $r$ such that for every $\varepsilon > 0$, we can choose $s>0$ such that the following holds. For every $p\in [\varepsilon,1]$, there is a coupling such that if $\calC^{p}_{G_q}(o')$ is infinite, then $\calC^{p,s}_{G}(o)$ is infinite. 
\end{proposition}
\begin{proposition}[Martineau-Severo]\label{Monoto2}
    Let $G$ be an infinite transitive graph with $p_c(G)<1$. Then for any choice of $r\geq 1$, the following holds: for every $s\in (0,1]$, there exists $p_c(G,s)<p_c(G)$ such that for every $p\in (p_c(G,s),1]$, the cluster $\calC^{p,s}_G(o)$ is infinite with positive probability.  
\end{proposition}

 Proposition~\ref{Monoto2} is Proposition 4.2 from \cite{MS19}. We now focus on the proof of Proposition~\ref{Monoto1}. But first, we want to show that these two propositions actually establish Theorem~\ref{Thm:Mono}.
\begin{proof}[Proof of Theorem~\ref{Thm:Mono} assuming Proposition \ref{Monoto1}]
First set $\varepsilon=p_c(G)/2$ and pick $s$ and $r$ according to Proposition~\ref{Monoto1}. By Proposition~\ref{Monoto2}, for $p>p_c(G,s)$ the cluster $\calC^{p,s}_G(o)$ is infinite with positive probability. By Proposition~\ref{Monoto1} we know that there is a coupling such that $\calC^{p}_{G_q}(o)$ is infinite with positive probability meaning $p\geq p_c(G_q)$. In particular, $p_c(G_q)\leq p_c(G,s)$ and we can conclude using $p_c(G,s)<p_c(G)$ from Proposition~\ref{Monoto2}.
\end{proof}

Recall that a set of vertices $A$ in $G$ is \textbf{$R$-dense} if any vertex can reach an element of $A$ in at most $R$ steps. 
\begin{lemma}{\label{lem:BouclePartition}}
    Let $G$ be an infinite transitive graph that is not a tree.
    Let $C=x_0,\dots,x_N=x_0$ be a cycle in $G$ with no repeating vertex. Then there exists a sequence $C_1,\dots,C_k,\dots$ of disjoint cycles and a partition $(P_i)_{i\in\N}$ of the vertices of $G$ with $C_i\subset P_i$ such that $C_k=\phi_k(C)$ with $\phi_k\in \Aut(G)$, whose union is $R$-dense for some $R>0$. In addition, it is also possible to choose $(P_i)$ such that there exists $M>0$ such that $|P_k|\leq M$ for all $k$.  
\end{lemma}
\begin{proof}
$\Aut(G)$ acts on the set of cycles in the graph $G$ and from the orbit of $C$, we can extract an infinite sequence of disjoint cycles that is maximal with respect to inclusion and set $A=\bigcup_{i\in \N}C_i$. If such a maximal set was not $R$-dense for any $R$, it would be possible to find in $G$ balls of arbitrary radius that do not intersect $A$. In particular, one could find a ball which radius is bigger than the length of the cycle. This would go against the maximality of $A$ as we could fit a new cycle isomorphic to $C$ within this ball. We therefore know that such a set is $R$-dense for some $R$. We take the minimal $R$ for this property.  
To create a partition from this sequence, we assign each vertex to the cycle $C_i$ it is closest to. When a vertex is at the same distance from different cycles, we arbitrarily decide it belongs to the same cell as one of its neighbors. We call $P_i$ the set of vertices associated with $C_i$ such that $V(G)=\bigcup_{i\in \N}P_i$. Note that $P_i\subset B_R(C_i)$ and $|B_R(C_i)|=|B_R(C)|$ so the cells have bounded cardinality. By construction, any vertex of $P_i$ at distance $d>0$ from $C_i$ has at least one neighbor in $P_i$ at distance $d-1$ from $C_i$ so $P_i$ is connected since $C_i$ is connected. 
\end{proof}
The next lemma is the one that determine a suitable value of $r$ depending on the geometry of the graph $G$. 
\begin{lemma}\label{ChoiceR}
    Let $G$ be an infinite transitive graph that is not a tree. Fix $q\in (0,1)$ and $(C_i)_{i\in\N}$ an $R$-dense sequence of disjoint cycles all isomorphic. There is a choice of $r\in \N$ such that for $x \in V(G_q)$,  $Z(x,r):=\pi^{-1}(B_r(\pi(x))$ contains $\pi^{-1}(C_i)$, with $x\in P_i$. There exists $p_0>0$ independent of $x$ such that $Z(x,r)$ is connected with probability at least $p_0$. 
\end{lemma}
\begin{remark}\label{propR}
    Note that by construction, an edge in $G_q$ is adjacent to $Z(x,r)$ if and only if its twin is also adjacent. The same property can be found in \cite{MS19}, despite a different definition of $Z(x,r)$.  
\end{remark}
\begin{proof}
Let $N=|C|$ be the length of the cycle. We know that there is a vertex in $C_i$ at distance at most $R$ from $\pi(x)$ meaning $B_r(\pi(x))$ with $r=R+N$ contains $C_i$ and $Z(x,r)$ contains $\pi^{-1}(C_i)$ The lift of $C_i$ is connected if and only if it is switching and we can calculate this probability $$c=\sum_{
\begin{array}{l}
      k=1  \\
      k \text{ odd}
\end{array}}^{N}\binom{N}{k}q^k(1-q)^{N-k},$$ which is positive and independent of $x$. We conclude by pointing out that if $C_i$ is connected, it is also the case for $Z(x,r)$. Let $y$ and $z$ be two vertices in $Z(x,r)$. Because $\pi(Z(x,r))=B(\pi(x),r)$ is connected, there is a path from to $\pi(z)$ to $\pi(y)$ that we can lift in $Z(x,r)$ starting at $z$. Either the lift of this path joins $z$ and $y$ and we are finished, or the lift joins $z$ and $T(y)$. We now need to prove that $y$ and $T(y)$ are connected. Let $u$ be a vertex in $C_i$. In $B(\pi(x),r)$, we can join $\pi(y)$ to $u$ follow the cycle $C_i$ back to $u$ and come back to $\pi(y)$. Lifting this path back in $G_q$ gives us a path that link  $y$ to $T(y)$ when $C_i$ is switching. We have therefore proved that any two vertices of $Z(x,r)$ are linked through a path in $Z(x,r)$ and so are connected with positive probability at least $p_0=c$. 
\end{proof}
\begin{remark}
Following the same proof, we know that since the $Z(x,r)$ are disjoint if the vertices $x$ are far enough, then for $q\in(0,1)$, $G_q$ is almost surely connected. 
\end{remark}
Before proving Proposition~\ref{Monoto2}, we need one last lemma that builds a set of independent Bernoulli variables that assures us a vertex can reach its twin in a bounded number of steps. 
\begin{lemma}{\label{lem:distrib-mag}}
Let $G$ be an infinite transitive graph that is not a tree. There exists $D$ in $\N$ such that, on a larger probability space, there exist $t>0$ and independent Bernoulli variables $\beta_x$ with parameter $t$ such that the following holds: $\beta_x=1$ implies that both lift of $x$ are connected in at most $D$ steps.
\end{lemma}
\begin{proof}
Take $(C_i)$ and $R$ that satisfy the conclusion of Lemma~\ref{lem:BouclePartition}. If we denote by $c$ the probability that the lift of  $C_i$ is switching, then the $\pi^{-1}(P_i)$ is connected with probability $p_i$ at least $c$. Given a configuration $(\eta_e)_{e\in E(G)}$. let us call $O_i$ the independent Bernoulli variables of parameter $c$ that track whether the lift of $C_i$ is connected. 

We then want to ``distribute'' the probability of connectivity to all vertices in $P_i$. To that purpose, we note the following fact:If $\mathcal{B}_1,\dots,\mathcal{B}_n$ are independent Bernoulli variables with parameter $1-(1-p)^{1/n}$ then $\mathcal{B}=\max(\mathcal{B}_1,....,\mathcal{B}_n)$ is a Bernoulli variable with parameter $p$. 
Reciprocally, by extending, if needed, our probability space we can suppose any Bernoulli variable $\mathcal{B}$ with parameter $p$ is the maximum of $N$ independent Bernoulli variables with parameter $1-(1-p)^{1/n} $. 

We use this decomposition with the variables $O_i$ and $n_i=|P_i|$ to assign to every vertex $x$ in $P_i$ a Bernoulli variable $V_x$  of parameter $1-\left(1-c\right)^{1/|P_i|}$ that is dominated by $O_i$. In addition, all $V_x$ are independent. Since the parameter of these Bernoulli variables may not be uniform between all cells (even though $|P_i|\leq |B_R(C)|=L$) we can multiply $V_x$ by another Bernoulli variable of parameter $\frac{1-\left(1-c\right)^{1/L}}{1-\left(1-c\right)^{1/|P_i|}}$ independent of everything done so far to get a Bernoulli variable $\beta_x$ with parameter $t=1-(1-c)^{1/L}>0$ and with the following property:The $(\beta_x)_{x\in G}$ are independent, identically distributed and if $\beta_x=1$ then $O_i=1$ and $\pi^{-1}(C_i)$ is switching with $x\in P_i$ (generally, we can always decrease the parameter of a Bernoulli variable from $a$ to $b$ by multiplying it with an independent Bernoulli variable of parameter $b/a$).

If the former is true, then by denoting $y$ a vertex of $C_i$, we can lift the shortest path from $x$ to $y$, walking along the cycle once and use the other lift of the path $y$ to $x$ to create a path between $x_0$ and $x_1$ in at most $D=2R+|C|$ steps. 
\end{proof}
Now that we have all the necessary lemmas, we can move on to the main proof.
\subsection{Coupling Proposition}

We  prove Proposition~\ref{Monoto1} by constructing step by step a coupling between a percolation on $G_q$ and an enhanced percolation on $G$. For completeness, we reproduce the totality of the argument from \cite{MS19}, with some changes due to the fact that the distance between a vertex and one of its twin is not deterministically bounded by a constant $D$. In our case, this bound is only probabilistically true meaning during the construction of the enhancement, we need to take into account the variables $\beta_x$ that track if a vertex can join its twin in less than $D$ steps for a specific $D$. During the process, we alternate steps of pure percolation and steps of enhanced percolation. The difference from \cite{MS19} is taken into account at steps of enhanced percolation. 

\begin{proof}[Proof of Proposition~\ref{Monoto2}]

Let us take $(C_i)$, $(P_i)$, $R$ and $r$ that satisfy the conclusion of Lemmas~\ref{lem:BouclePartition} and~\ref{ChoiceR}.
Let $M=|B_{r+2}(u)|+1$ for some vertex $u$ of $G$ (because $G$ is transitive, the choice of $u$ does not matter). For any graph $H$, we define the multigraph $\hat{H}$, to have the same vertex set $H$ and its edge set to be $E(H)\times \{1,...,M\}$ where $(\{x,y\},k)$ is an edge connecting $x$ and $y$. Percolation on $\hat{H}$ is done with parameter $\hat{p}=1-(1-p)^{1/M}$ so that percolation on $\hat{H}$ with parameter $\hat{p}$ is equivalent to percolation on $H$ with parameter $p$. 

Let $\eta = (\eta_e)_{e\in E(G)}$ describe the switching state of   edges in $G_q$. Setting a specific $\eta$ allows us to specify the variables $\beta_x$ described in Lemma~\ref{lem:distrib-mag}. Let us take $\omega$ a configuration for percolation on $\hat{G_q}$ with parameter $\hat{p}$.
We are going to describe an exploration of $\hat{G}$ and $\hat{G_q}$ creating two other variables $\kappa=(\kappa_e)_{e\in E(\hat{G})}$ and $\alpha=(\alpha_x)_{x\in V(G)}$ that follow the law of augmented percolation on $G$. They will be coupled in such a way, that the infinite cardinality of the cluster of the origin in $G_q$ with $\omega$ will imply the infinite cardinality of the cluster of augmented percolation in $G$ with $(\kappa,\alpha)$. 

Edges will be explored in two possible ways:$p$-exploration and s-exploration. $p$-exploration corresponds to the usual exploration in Percolation Theory, and s-exploration means the edge has been used in an ``Enhancement property''. Vertices may also be $s-$explored for the same reason. We say an edge is explored if it is either $s$-explored or $p$-explored. No edges will be both $p-$ and $s-$explored. The cluster of a vertex $o$ in bond percolation (resp. augmented bond percolation) in the graph $H$ will be written as $\calC^p_H(o)$ (resp. $\calC^{p,s}_H(o)$). 

Along the way, we will describe the sets $C_{\infty}\subset V(G)$, that follows the law of $\calC_{G}^{p,s}$, and $C'_{\infty}\subset V(G_q)$ that is stochastically dominated by $\calC_{G_q}^p(o)$.These sets will be defined as infinite unions of $C_{\ell,n}$ and $C'_{\ell,n}$, which will correspond to the vertices explored ``at each step''. To increment the variable $\ell$ by $1$, we will have to make a certain (potentially infinite) number of increments in $n$. From now on, we will refer to ``step'' as the set of action made in order to increment $\ell$ by one. On odd steps, we will simply explore the connected clusters of some set of vertices, just like in usual percolation. On even steps, however, we will exclusively explore our graph using the special property of augmented percolation. We decide a well-ordering of the vertices and edges of $G$ and $G_q$. In the following, by ``pick an element $a$ such that $P$'', we mean ``take the smallest $a$ such that $P$ holds''. 
\vspace{5mm}\\
\textbf{Properties of the process:}

\begin{enumerate}[label=\Alph*)]
    \item If an edge $e$ in $E(G)$ is $p$-explored, then there is a lift $e'$ of $e$ in $G_q$ such that the set of the $p$-explored lifts of $e$ is precisely $\left\{e'\right\} \times\{1, \ldots, M\}$.
    \item If an edge $e$ in $E(G)$ is $p$-unexplored, then all of its lifts are unexplored.
    \item Every element of $C'_{\ell, n}$ is connected to $o'$ by an $\omega$-open path.
    \item  For every edge $e$ in $G$ and each lift $e'$ of $e$ in $G_q$, the number of $s$-explored edges of the form $\left(e', k\right)$ is at most the number of vertex at distance at most $r+1$ from some endpoint of $e$ and is therefore always smaller than $M$. This property is the reason for our choice of $M$.
    \item The map $\pi$ induces a well-defined surjection from $C'_{\ell, n}$ to $C_{\ell,n}$.
\end{enumerate}
\textbf{Step 0 :}  Initially, no edge is either $p$-explored or $s$-explored and $C_0=\{o\}$ and $C'_0=\{o'\}$.\vspace{4mm}\\
\textbf{Step 2N+1:} Set $C_{2N+1,0}=C_{2N}$ and $C'_{2N+1,0}=C'_{2N}$. This step persists as long as there is an unexplored edge in G such that one if its endpoint lies in $C_{2N+1,n}$. 
\begin{enumerate}
    \item Take $e=\{u,v\}$ to be the smallest such edge with $u\in C_{2N+1,n}$.
    \item Pick $e'$ some lift of $e$ intersecting $C_{2N+1, n}'$.
    \item Declare $e$ and all ($e',k$) to be $p$-explored (they were unexplored before because of conditions A and B). 
    \item For every $k \leq M$, define $\kappa_{\left(e, k\right)}=\omega_{(e', k)}$.
    \item Set $\left(C_{2N+1, n+1}, C_{2N+1, n+1}'\right)=\left(C_{2N+1, n}, C_{2N+1, n}'\right)$ if all the $(e', k)$  are closed; otherwise, set $C_{2N+1, n+1}, =C_{2N+1, n} \cup\{v\}$ and $C'_{2N+1, n+1}=C'_{2N+1, n}\cup e'$.
\end{enumerate}
\vspace{1mm} When this step is finished, which occurs after finitely or countably many iterations, set $C_{2N+1}=\bigcup_{n} C_{2N+1, n}$ and $C'_{2N+1}=\bigcup_{n} C'_{2N+1, n}$.\vspace{4mm}
\newline
\textbf{Step 2N+2:} Set $C_{2N+2,0}=C_{2N+1}$ and $C'_{2N+2,0}=C'_{2N+1}$.
Check that at least one $s$-unexplored vertex in $C_{2N+1}$ whose $r$-ball is in $C_{2N+1}$ and fully open (meaning for every edge that lie in the $r$-ball, at least one of its copy is open) in $\kappa$. If it is the case, do the following (otherwise, finish this step):

\begin{enumerate}
    \item Take $u$ to be the smallest such vertex.
    \item Pick  $x \in C_{2N+1}' \cap \pi^{-1}(\{u\})\neq\emptyset$. 
    \item For each $p$-unexplored edge $e'$ in $Z(x, r)$, take its $s$-unexplored copy $\left(e', k\right)$ in $\hat{G_q}$ of smallest label $k$, and switch its status to $s$-explored. It is possible because a multi-edge can get $s$-explored less than $|B_{r+2}(u)|$ times during the exploration and $M$ has been set specifically larger than this amount by (D).
    \item If all these newly $s$-explored edges are open (so that $Z(x,r)$ is ``fully open''), then perform this substep. By remark \ref{propR}, for every $\mathcal{G}$-edge $e \in S_{r+\frac{1}{2}}(u)$,
    both of its lifts are adjacent to $Z(x, r)$ and by (A) one of them is $p$-unexplored. By (D) and the value of $M$, one of its copies $\left(e', k\right)$ is $s$-unexplored: pick that with minimal $k=: k_{e}$. Declare all these edges to be $s$-explored. If all these $\left(e', k_{e}\right)$ are open, then say that this substep is successful.
    \item Recall the variable $\beta_x$ that is associated with the vertex $x$ is a parameter $t$ and independent of others $\beta_y$. We simply say this substep is successful if $\beta_x=1$ meaning we are sure that both lift of $x$ are connected in at most $D$ steps in $Z(x,r)$ by construction.
    %\item Note that conditioned on everything before step 2N+2, the probability $u_x$ that the last steps are successful can be bounded from below by $s=\hat{p}^{|B_{r}(u)|+|S_{r+1/2}(u)|}\times t$ which is independent of $u$. To make this probability uniform, we take $\gamma_u$ a random Bernoulli variable of parameter $u_x/s$ and say this step is successful if the variable equals 1. 
    \item Conditioned on every step being successful we set $\alpha_u=1$. Otherwise, set $\alpha_u=0$ if one or more step failed (therefore $\alpha_u$ is a Bernoulli variable with parameter $s$ and the $\alpha_u$ variables are independent).
    \item If $\alpha_{u}=1$, then set $C_{2N+2, n+1}:=C_{2N+2, n} \cup S_{r+1}(u)$ and $C'_{2N+2, n+1}$ to be the union of $C_{2N+2, n}, Z(x, r)$, and the lifts of the elements $S_{r+1/2}(u)$ in $G_q$ that were chosen at substep $4.$. Notice that condition (C) continues to hold as in this case $Z$ is ``fully-open''. Otherwise, set $C_{2N+2, n+1}:=C_{2N+2, n}$ and $C_{2N+2, n+1}^{\prime}:=C'_{2N+2, n}$.
\end{enumerate}
When this step is finished, set $C_{2N+2}=\bigcup_{n} C_{2N+2, n}$ and $C_{2N+2}^{\prime}=\bigcup_{n} C_{2N+2, n}^{\prime}$.
\vspace{1mm}\\
\textbf{Step $\infty$ :} Set $C_{\infty}:=\bigcup_{N} C_{N}$ and $C_{\infty}^{\prime}:=\bigcup_{N} C_{N}^{\prime}$. Wherever $\kappa$ is undefined, define it as a random Bernoulli variable of parameter $\hat{p}$. In the same way, wherever $\alpha$ is undefined, toss independent Bernoulli random variables of parameter $s$, independent of everything done so far.
\vspace{5mm} \\
By construction, $C_{\infty}$ has the distribution of the cluster of the origin for the augmented percolation on with parameter $(p, s)$ on $\mathcal{G}$:it is the cluster of the origin of $\left(\left(\bigvee_{k} \omega_{e, k}\right)_{e}, \alpha\right)$ which has distribution $B(p)^{\otimes E(G)} \otimes B(s)^{\otimes V(G)}$. Besides, $C'_{\infty}$ is included in the cluster of $o'$ for $\left(\bigvee_{k} \omega_{e, k}\right)_{e}$, which is a $p$-bond percolation on $G_q$. Finally, the coupling guarantees that $\pi$ surjects $C'_{\infty}$ onto $C_{\infty}$ and therefore if $C_{\infty}$ has infinite cardinality it is also the case for $C'_{\infty}$ 
\end{proof}

\section{Sharpness at $q=1/2$}\label{Sec:Sharp}
Our goal in this section is to prove Theorem \ref{Thm:Sharp}.
It states that the tail size of the cluster of any given vertex, decreases exponentially in the subcritical regime. To prove Theorem \ref{Thm:Sharp}, we closely follow Vaneuville's strategy in \cite{Van23}. The difference between Vaneuville's case and ours is the change to the definition of an exploration and the addition of a new crucial lemma that is Lemma \ref{lem:trou_sto}. Our new definition of an exploration explores edges as well as vertices to partially determine the structure of the random graph. Lemma \ref{lem:trou_sto} was just an inequality from transitivity in \cite{Van23} but needed to be extended to a full lemma in our random graph. Lemma \ref{lem:trou_sto} is the only place where the hypothesis $q=1/2$ is needed. 

Note that the exponential decay is here stated by averaging over all the possibles values of $\eta$. A quenched version of Theorem \ref{Thm:Sharp} is 
\begin{theorem}\label{thm:exp-quenched}
    Let $G$ be a transitive infinite graph and let $\nu$ be the random graph measure of $G_{1/2}$. For every $p<p_c(1/2)$, there exists $c>0$ such that for $\nu$-almost every $\eta$, there exists $C>0$ such that
\[
\forall n \ge 0, \quad \Pro_p[|\calC_o|\geq n|\;\eta\;] \leq Ce^{-cn}.
\]
\end{theorem}
Although these are two different results, we can show that the annealed version implies Theorem \ref{thm:exp-quenched}. 
\begin{proof}[Proof of Theorem \ref{thm:exp-quenched} using Theorem \ref{Thm:Sharp}]
For $p<p_c(1/2)$, let us take $c$ and $C$ given by Theorem \ref{Thm:Sharp}. By the Markov inequality, we have 
$$\nu\Big(\Pro_p[|\calC_o|\geq n|\;\eta\;] \geq Ce^{-cn/2}\Big)\leq \Pro_p[|\calC_o|\geq n|]C^{-1}e^{cn/2}\leq C^{-1}e^{-cn/2}.$$ which is summable in $n$. By the Borel--Cantelli Lemma, it means that the exponential decay happens for $c/2$ and $C^{-1}$ almost surely for all but finitely many natural numbers. We can take a larger $C'$ to obtain a result for all $n\geq 0$.
\end{proof}

To prove Theorem \ref{Thm:Sharp}, we will follow the strategy of \cite{Van23}. We will introduce a ghost field $\calM$ to be a random subset of the vertex set $V(G_q)$ where each element $m\in V(G_q)$ has probability $1-e^{-h}$ of being in $\calM$, independently of each other. We denote $\Pro_{p,h}$ the product measure of this distribution with the percolation on the random graph $G_q$.
Let,
$$\psi_n(p) = \Pro_p\big[|\calC_o| \ge n\big] \quad $$
and 
$$m_h(p)=\Pro_{p,h} \big[ \calC_o \cap \calM \ne \emptyset \big].$$
Our inquiry will lead us to prove the following theorem, of which the main result of this section is a direct consequence.  
\begin{theorem}\label{thm:exp}
Let $G$ be a transitive infinite graph and take $p\in(0,1)$ and $h\in(0,+\infty)$. Then, letting $s=p(1-2m_h(p))$, we have
\[
\forall n \ge 0, \quad \psi_n(s) \le \frac{1}{1-m_h(p)} \psi_n(p) e^{-hn}.
\]
\end{theorem}
\begin{proof}[Proof of Theorem \ref{Thm:Sharp} using Theorem \ref{thm:exp}]
Let $p<p'$ be two positive numbers strictly smaller than $p_c$. Since we are in the subcritical regime with $p<p_c(1/2)$, the cluster of the origin is almost surely finite and therefore $m_h(p')\rightarrow 0 $ as $ h$ goes to $+\infty$. We can then find $h$ small enough such that $p'(1-2m_h(p'))>p$. For such a $p'$, we have by Theorem \ref{thm:exp} $$\psi_n(p)\leq \psi_n\big(p'(1-2m_h(p'))\big)\leq \frac{\psi_n(p')}{1-m_h(p')}e^{-hn}\leq \frac{1}{1-m_h(p')}e^{-hn}$$ which concludes our proof with $c=h$ and $C=\frac{1}{1-2m_h(p')}$.
\end{proof}
Although we prove the exponential decay only in the case $q=1/2$, we will keep the notation $G_q$ throughout this section, since most of our intermediate results still hold for $q\in[0,1]$. We will highlight the importance of $q=1/2$ when necessary.

\subsection{Random graphs and exploration}

Exploring a graph in percolation means looking at the state of different edges one by one, usually only following the connected component of a specific vertex. One problem that arises when exploring random graphs, is that in order to make a decision on which edge to explore next, we need to know which edge are in our realization of the random graph. In this context, an exploration will need to make decisions based on the openness of edges but also on the partial graph structure discovered previously. We need to make a precise decision on what information to reveal when exploring our graph. We need to distinguish between exploring a vertex and exploring an edge. One consequence is that we may explore an edge without knowing one of its endpoint. Our proof strategy will need a crucial property (that is Lemma \ref{lem:trou_sto}) that in transitive deterministic graph, is related to the BK-inequality : If one explores parts of a transitive graph $G$ then cuts this part of the graph, the result is a new graph where percolation is harder than in the original graph $G$. However, this property may not hold with percolation on random graphs, as exploring the cluster of a vertex also partially prescribe its geometry. Therefore, a particularly lucky geometry may favor percolation even if a part of the graph has been cut.

To account for this difficulty, we need to change our definition of a graph. Let $V$ and $E$ be two sets known as the vertex set and the edge set respectively. Let $P_2(V)$ denote the set composed of subsets of $V$ with $2$ elements. The graph structure is embedded in a function $f:E\to P_2(V)$ such that $f(e)=\{x,y\}$ if the edge $e$ links $x$ and $y$. Likewise, we also make use of a function $g:V\to P(E)$ such that $e\in g(x) $ iff $x\in f(e)$.  Given a function $f$, there is only one corresponding $g$ function (but a given $g$ function may not correspond to any $f$ function). Therefore, we will focus only on the function $f$ in our definitions. Note that whenever a function $f$ is defined, we consider the function $g$ to be implicitly defined as well. Note that this definition resembles Serre's definition of a graph.
 
In the case of our random graphs, the randomness of the graph structure purely lies in the $f$ and $g$ functions, while the edge and vertex set are fixed. Therefore, a random graph is defined as a triplet $\calH=(V,E,\nu)$ where $\nu$ is a probability measure on the set $\cal{F}$ of functions from $E$ to $P_2(V)$. A realization of the random graph $\calH$ is a certain function $f:\,E\to P_2(V) $. Throughout this section, $\nu$ will always denote the measure that controls the graph structure of $G_q$ but may refer to either the version with $\eta$ variables or with the functions $f$.

Let us see how we adapt our new definition to the case of $G_q$ constructed before from a transitive deterministic graph $G$. Let $e=\{x,y\}$ be an edge in $G$ and $\eta_e\sim \calB(q)$ its switching Bernoulli variable. Recall that $x_0,x_1,y_0$ and $y_1$ are the lifts of $x$ and $y$ in $G_q$. We consider the edges $e_0$ and $e_1$ as abstract objects such that 
\begin{itemize}
    \item If $ \eta_e=0$ then we choose one of the following cases with probability 1/2
$\left\{
\begin{array}{ll}
     f(e_0)=\{x_0,y_0\},\,f(e_1)=\{x_1,y_1\}\\
     f(e_1)=\{x_0,y_0\},\,f(e_0)=\{x_1,y_1\}
\end{array}
\right.$
    \item If $ \eta_e=1$ then we choose one of the following cases with probability 1/2 \quad
$\left\{
\begin{array}{ll}
     f(e_0)=\{x_0,y_1\},\,f(e_1)=\{x_1,y_0\}\\
     f(e_1)=\{x_0,y_1\},\,f(e_0)=\{x_1,y_0\}
\end{array}
\right.$
\end{itemize}
Informally, if $ \eta_e=0$, we know that $x_0$ et $y_0$ are going to be linked but determine randomly if the edge linking them is $e_1$ or $e_0$. The same idea holds for the vertices $x_1$ and $y_1$. In the same way, if $\eta_e=1$, we know that $x_0$ et $y_1$ are going to be linked but determine randomly if the edge linking them is $e_1$ or $e_0$. Conditionally to $\eta$, the choices are made independently. By considering every edge in $G$, we can fully describe $f$ and $g$ with the edge set as $E=\bigcup_{e\in E(G)}\{e_0,e_1\}$. 
\\

For a finite edge set $E$, we denote $\vec{E}$ the set of all orderings of $E$, that is all the elements $(e_1,...,e_{|E|})\in E^{|E|}$ such that each edge appear exactly once. 

\begin{definition}\label{defi:expl}
An \textbf{exploration} of $\calH$ is a map
\[
\begin{array}{rl}
\ee:\{0,1\}^{E}\times F& \longrightarrow V^{|E|}\times \vec{E}\\
\omega,f & \longmapsto (\xx_1,...,\xx_{|E|}),(\ee_1,\dots,\ee_{|E|})
\end{array}
\]
such that \begin{itemize}
    \item $\xx_1$ does not depend on $(\omega,f)$
    \item $\xx_k$ depends only on $\xx_1,\dots,\xx_{k-1}$ (possibly through $g$), $\ee_1,\dots,\ee_{k-1}$ (possibly through $f$) and $\omega_{\ee_1},\dots,\omega_{\ee_{k-1}}$
    \item $\ee_k$ depends only on $\xx_1,\dots,\xx_{k}$ (possibly through $g$), $\ee_1,\dots,\ee_{k-1}$ (possibly through $f$) and $\omega_{\ee_1},\dots,\omega_{\ee_{k-1}}$
\end{itemize}
\end{definition}
\begin{remark}
 Exploring in this context should be thought of as alternating between revealing properties of vertices (which edge it is connected to) and properties of edges (which are of two types:its endpoints and its openness status).
\end{remark}

Given an exploration $\ee$, $(e,x,\omega_0,f_0) \in \vec{E}\times V^{|E|} \times\{0,1\}^E\times \calF$ and $k \in \{0,\dots,|E|\}$, we denote by $\Expl_k(e,x,\omega_0,f_0)$ (or only $\Expl_k$) the event that $\ee$ on $(\omega,f)$ coincides with $(e,x)$ at least until step $k$, i.e.\
\begin{multline*}
\Expl_k(e,x,\omega_0,f_0)=\Big\{f,\omega \in \calF\times\{0,1\}^E:\forall j \in \{1,\dots,k\},g(\xx_j)=g_0(\xx_j)\\ \text{ and }\omega_{\ee_j}=({\omega_0})_{\ee_j}\Big\}
\end{multline*} 
We also let $\Expl(e,x,\omega_0,f_0)=\Expl_{|E|}(e,x,\omega_0,f_0)$.
\begin{remark}\label{rem:crucial 1}
    Note that the information contained within $\Expl_k$ breaks the symmetry between $f$ and $g$ as we only ask that the value of $g$ coincidence with $g_0$, the edges $e_i$ are partly undetermined. Depending on the chosen distribution for $f(e_j)$, the knowledge of $\Expl_k$ may not contain the knowledge of all the neighbors of $x_j$. This fact will become crucial later. 
\end{remark}
\medskip
\subsection{Preliminary lemmas}
In this subsection, $\calH=(V,E,\nu)$  
will always denote a finite random graph. As we need to consider both randomness from the graph structure and from percolation, we will say that an event $A\subset {\{0,1\}^E\times\calF}$ is \textbf{$p$-increasing} if for every $(\omega,f)\in A$ then $(\omega',f)\in A$ whenever $\omega\leq \omega'$. If $\mu$ and $\mu'$ are two probability measures on ${\{0,1\}^E\times\calF}$, we will say that $\mu'$ \textbf{$p$-dominates} $\mu$, written $\mu\preceq\mu'$, when $\mu(A)\leq \mu'(A)$ for all $p$-increasing events $A$. We fix $\ee$ an exploration of $\calH$ a random finite graph associated with a probability measure $\nu$.

\begin{lemma}\label{lem:gen}
Let $q \in [0,1]$ and $\calH=(V,E,\nu)$ be a finite random graph. Let $\mu$ be a probability measure on $\{0,1\}^E$. Write $ \mu^{\calH}=\mu\otimes\nu$ and assume that for every $(e,x,\omega_0,f_0) \in \vec{E}\times V^E\times\{0,1\}^E\times \calF$ such that $\mu^{\calH}\big[\Expl\big]>0$, we have
$$
\forall k \in \{0,\dots,|E|-1\}, \quad \mu^{\calH} \big[ \omega_{e_{k+1}} = 1 \; \big| \; \Expl_k \big] \geq p.\qquad (*)
$$
Then, $\Pro^{\calH}_p \preceq \mu^{\calH}$ where $\Pro^{\calH}_p=\Pro_p\otimes \nu$ and $\Pro_p$ denotes Bernoulli percolation with parameter $p$. 
\end{lemma}
To keep things readable, we drop the superscript $\calH$ for our measure $\mu$. Therefore, $\mu$ indiscriminately means a probability measure on $\{0,1\}^E$ and its product by the random graph measure. 
\begin{proof}[Proof of Lemma \ref{lem:gen}]
We prove the lemma by induction on $|E|$. If there are no edges in $\calH$, the result is immediate.
When $\calH$ has at least one edge, let $A\subset \{0,1\}^E\times F $ be a $p$-increasing event. We fix a vertex $x$ and let $(a_i)_{i\in I}$ describe all the possible values for $g(x)$. We can then write 
\begin{align}
     \mu(A)=&\sum\limits_{i\in I}\mu(g(x)=a_i) \mu_i(A) \label{eq:probtot}\\
     =& \sum\limits_{i\in I}\mu(g(x)=a_i) \Big[\mu_i(A|\omega_{e_1}=1)\mu_i(\omega_{e_1}=1) + \mu_i(A|\omega_{e_1}=0)\mu_i(\omega_{e_1}=0)\Big].\nonumber
\end{align}
where $\mu_i=\mu(\,\cdot\,|\,g(x)=a_i)$.

The knowledge of $g(x)$ and $\omega_{e_1}$ is the knowledge of $\Expl_1$ so the hypothesis $(*)$ carries over to 
$\mu_i(\,\cdot\,|\,\omega_{e_{1}}=\varepsilon)$. 
Also, $\mu_i(\,\cdot\,|\,\omega_{e_1}=\varepsilon)$ is still a product measure between a percolation on $\{0,1\}^{E\setminus\{e_1\}}$ and the graph measure on $G$.
% Un poil sloppy ici. 
Therefore, let 
\begin{multline*}
A^{\varepsilon,i}=\Big\{\omega^\star,f^\star\in \{0,1\}^{E\backslash \{e_1\}}\times\calF\,|\,g^\star(x)=a_i\text{ and }\forall \omega'\in \{0,1\}^E\\\,\big((\omega'_{e_1}=\varepsilon\text{ and }\omega'=\omega^\star\text{ outside of }e_1)\Rightarrow (\omega',f^\star)\in A\big)\Big\}.
\end{multline*}
By construction, we have   $\mu_i(A|\omega_{e_1}=\varepsilon)=\mu_i(A^{\varepsilon,i}|\omega_{e_1}=\varepsilon)$. Let us write $\PR_{p,a_i}=\Pro_p(\,\cdot\,|g(x)=a_i)$ so by the induction hypothesis we have 
\begin{multline}\label{eq:ineq1}
\mu_i(A|\omega_{e_1}=1)\mu_i(\omega_{e_1}=1) + \mu_i(A|\omega_{e_1}=0)\mu_i(\omega_{e_1}=0)\\ 
\geq \Pro_{p,a_i}(A^{1,i})\mu_i(\omega_{e_1}=1) + \Pro_{p,a_i}(A^{0,i})\mu_i(\omega_{e_1}=0).
\end{multline}
The event $A$ is $p$-increasing so $A^{\varepsilon, i}$ is also $p$-increasing. We have $A^{0,i}\subset A^{1,i}$ and as a result $\Pro_{p,a_i}(A^{0,i})\leq \Pro_{p,a_i}(A^{1,i})$.  Because $g(x)$ is independent of the $\omega_e$ variables, we have $\mu_i(\omega_{e_1}=1)=\mu(\omega_{e_1}=1)\geq p$ so we can write:
\begin{multline}\label{eq:ineq2}
\Pro_{p,a_i}(A^{1,i})\mu_i(\omega_{e_1}=1) + \Pro_{p,a_i}(A^{0,i})\mu_i(\omega_{e_1}=0)\\\geq \Pro_{p,a_i}(A^{1,i})p + \Pro_{p,a_i}(A^{0,i})(1-p)=\Pro_p(A\,|\,g(x)=a_i).
\end{multline}
which concludes our proof as plugging (\ref{eq:ineq1}) and (\ref{eq:ineq2}) in (\ref{eq:probtot}) gives us
$$\mu(A)\geq\sum_{i\in I}\mu(g(x)=a_i)\,\Pro_p(A\,|\,g(x)=a_i)=\Pro_p(A).$$
\end{proof}
We now add a new layer of randomness to our model. For $h\in [0,1]$, we denote by $\sigma_h$ the product of Bernoulli measures with parameter $h$ on every vertex in $\calH$. If $\calM\subset V$ is a random subset given by $\sigma_h$, we will call the elements of $\calM$ the green vertices. We now consider the product measure $\Pro_{p,h}=\Pro^{\calH}_p\otimes\sigma_h$. Note that for the next lemma, the properties of the measure $\sigma_h$ are not  important. 

Recall that an edge $e$ is called \textbf{$p$-pivotal} in a configuration $(\omega,f,S)$ for an event $A\subset \{0,1\}^E \times \calF\times \{0,1\}^V$ if changing  the value of $\omega_e$ changes whether $(\omega,f,S)$ belongs to $A$. 
\begin{lemma}\label{lem:coupl}
Let $\calH$ be a finite random graph and let $p \in (0,1)$ and $h \in (0,+\infty)$. Moreover, let $A \subseteq \{0,1\}^E \times \calF\times \{0,1\}^V$ be any non-empty set and let $\varepsilon \in [0,1]$. Assume that for every $(e,x,\omega_0,f_0) \in \vec{E}\times V^E\times\{0,1\}^E\times \calF$  such that $A \cap \Expl\neq \emptyset$, we have
\[
\forall k \in \{0,\dots,|E|-1\}, \quad \Pro_{p,h} \big( \text{$e_{k+1}$ is pivotal for $A$} \; \big| \; A \cap \Expl_k\big) \le \varepsilon.
\]
Then,
\[
\Pro_{p(1-\varepsilon)}\preceq \Pro_{p,h} \big( (\omega,f) \in \cdot \; \big| \; A \big).
\]
\end{lemma}
Note that the last line of Lemma \ref{lem:coupl} is a \textbf{$p$-domination} for probability measures on $\{0,1\}^E \times \calF$.  Since we have extended our probability space again, we should be using $\Expl\times \{0,1\}^V$ instead of $\Expl$ but we omit it for readability reasons. 

\begin{proof}[Proof of Lemma \ref{lem:coupl}]
To prove stochastic domination, we use the Lemma \ref{lem:gen} with $\mu=\Pro_{p,h}\big[ (\omega,f) \in \cdot \; \big| \; A \big]$. We now therefore need to bound the quantity $\Pro_{p,h}(\omega_{e_{k+1}}=1|A\cap \Expl_k)$ from below. 

\begin{multline*}
     \Pro_{p,h}(\omega_{e_{k+1}}=1|A\cap \Expl_k)=\frac{\Pro_{p,h}\big((\omega_{e_{k+1}}=1)\cap A|\Expl_k\big)}{\Pro_{p,h}(A|\Expl_k)}  \\
      \geq \frac{\Pro_{p,h}\big((\omega_{e_{k+1}}=1)\cap A\cap (e_{k+1}\text{ is not piv. for }A)\big|\Expl_k)}{\Pro_{p,h}(A|\Expl_k\big)}\\
     =p\times \Pro_{p,h}(e_{k+1}\text{ is not piv. for }A|A\cap\Expl_k)\geq p(1-\varepsilon).
\end{multline*}

The last equality stems from the fact that when $e_{k+1}$ is not pivotal for $A$, $\omega_{e_{k+1}}$ is independent of the $sigma$-field generated by $A$ and the event $\Expl_K$ (as is the event “$\omega_{e_{k+1}}$ is not pivotal for $A$”). We therefore conclude with Lemma \ref{lem:gen} with the parameter $p(1-\varepsilon)$. 
\end{proof}

\subsection{ Remaining graph}
We now go back to our original random $G_q$ created by switching pairs of edges in two copies of $G$. However, since we have only considered finite graphs in the previous section, we restrict our view to a subgraph $G^n_q$. We define $G_q^n$ as the subgraph of $G_q$ induced by the set of vertices $\pi^{-1}(V_n)$, where $V_n$ is the set of vertices at distance at most $n$ from the origin in $G$. Alternatively, $G^n_q$ can be defined as $(G^n)_q$ by taking the ball of radius $n$ around the origin in $G$ and creating its random 2-lift $G^n_q$.

\begin{lemma}\label{lem:coupl_perco}
Let $p \in (0,1)$ and $h \in (0,1)$, and write $s=p(1-2m_h(p))$. Then,
\[
\forall n \ge 0, \quad \Pro^{G^n_q}_{s} \preceq \Pro^{G^n_q}_{p,h} \big( (\omega,f) \in \cdot \; \big| \; \calC_o \cap \calM = \emptyset \big).
\]
\end{lemma}
Theorem \ref{thm:exp} is a direct consequence of this lemma, as proven by the following calculation. 
\begin{proof}[Proof of Theorem \ref{thm:exp} using Lemma \ref{lem:coupl_perco}]
Recall that $\psi_n(s)=\Pro_{s}(|\calC_o|\geq n])=\Pro_{s}^{G_n}(|\calC_o|\geq n)$ , we have by Lemma \ref{lem:coupl_perco}
\begin{align*}
\psi_n(s)\le\Pro_{p,h}^{G_n} \big( |\calC_o|\ge n  \big|  \calC_o \cap \calM = \emptyset
 \big) &\overset{\text{Bayes}}{=} \frac{\psi_n(p) \times \Pro_{p,h}^{G_n}\big(\calC_o \cap \calM = \emptyset \; \big| \; |\calC_o| \ge n \big)}{1-\Pro_{p,h}^{G_n}\big( \calC_o \cap \calM \neq \emptyset \big)}\\
& \hspace{0.2cm}=  \frac{\psi_n(p) \times \E_{p,h}^{G_n}\big[ e^{-h|\calC_o|} \; \big| \; |\calC_o| \ge n \big]}{1-\Pro_{p,h}^{G_n}\big( \calC_o \cap \calM \neq \emptyset \big)}\\
& \hspace{0.2cm}\le \frac{\psi_n(p)}{1-m_h(p)} e^{-hn}
\end{align*} %We finish our proof by begging for mercy to the notation god. 
\end{proof}

Proving Lemma  \ref{lem:coupl_perco} will require us to introduce the notion of \textbf{remaining graph} defined below.
\begin{definition}
    Given a random graph $\calH=(V,E,\nu)$, a connected set of vertices $\calC_o\subset V$ and a function $g_o:\calC_o\to P(E)$, we denote by $E_0\subset E$, the set of edges with at least one endpoint in $\calC_o$ according to $g_o$ and we call $$\widetilde{\calH}=\Big(V\setminus \calC_o,E\setminus E_o,\nu\big(\,\cdot\,|\,\forall x\in\calC_o,g(x)=g_o(x)\,\big)\Big)$$ the remaining graph after $\calC_o$.
\end{definition}
Note that while $\calC_o$ is deterministic in our definition, it can be thought of as a random cluster of the origin. We will write $\widetilde{\pi}$ for the projection from $\widetilde{G_q}$ to $G$. The need for the notion of remaining graph will appear in the proof  of Lemma \ref{lem:coupl_perco}. It will boil down to the following inequality. 

\begin{lemma}\label{lem:trou_sto}
   Let $G$ be an infinite transitive graph. 
    Let $\calC_o$ be any set of vertices of $G_q$ that contains $o$ et and function $g_0:\calC_o\to P(E)$ that makes $\calC_o$ connected. Let us denote $\widetilde{\calC_x}$, the cluster of the vertex $x$ in $\widetilde{G_q}$. 
    For $q=1/2$ and any $p<p_c(1/2)$, we have almost surely for the remaining graph $\widetilde{G_q}$ the inequality:$$\Pro_p^{\widetilde{G_q}}(\widetilde{\calC_x}\cap\calM\neq\emptyset)\leq \Pro_p^{G_q}(\calC_x\cap\calM\neq\emptyset)$$ for any vertex $x$ outside of $\calC_o$. 
\end{lemma}
We will prove this lemma by coupling percolation on $G_q$ and on $\widetilde{G_q}$. The result will only be a stochastic domination of the size of clusters (and not directly of clusters) as we may ``break'' the geometry of $\widetilde{G_q}$ while constructing our coupling.

To make sense of this lemma, let us look at the equivalent property on a non-random transitive graph. Here, the remaining graph after exploration is just the graph $G$ where we deleted the connected component of the origin. On this subgraph $\widetilde{G}$, the cluster of any vertex is necessarily stochastically smaller in $\widetilde{G}$ than $G$ by transitivity. It is therefore harder to reach a green vertex from there. 
The main issue that arises when dealing with random graphs is that our exploration of $\calC_0$ not only discover the openness of edges in (and on the border of) $\calC_o$ but also partially prescribe the geometry of the graph. We could then imagine that despite the deletion of edges in $\calC_o$, an extremely unlucky geometry would favor exploration from a given vertex violating the inequality of Lemma \ref{lem:trou_sto}. What Lemma~\ref{lem:trou_sto} ensures is that if we take $q=1/2$, this latter case does not occur.

\begin{proof}[Proof of Lemma \ref{lem:coupl_perco} given Lemma~\ref{lem:trou_sto}:]
We will apply Lemma \ref{lem:coupl} with $\varepsilon=2m_h(p)$ and $A=(\calC_o\cap\calM = \emptyset)$. Let us fix an arbitrary well-ordering of the vertices and the edges of $G^n_q$. 
 We define an exploration of $G^n_q$ in the following way:We start by revealing $\calC_0$, the cluster of the origin $o$.
More precisely, we start by revealing the edges connected to the origin. Then we start applying the following rule:
 We pick the smallest edge $e$ yet to be explored and connected to our cluster and reveal its state. If it is closed, we reveal the neighbors of the smallest vertex already explored. 
 If it is open, we first reveal $f(e)$ to know both endpoints of $e$. Let $y$ be the endpoint of $e$ that is not yet in the explored connected component (or take the smallest endpoint if both were already explored). Then we reveal the neighboring edges of $y$. We repeat the process until we have explored all of $\calC_o$. Once we are done exploring the cluster of the origin, we reveal the status of the remaining vertices and edges following our ordering. 
\begin{remark}\label{rem:crucial 2}
One important aspect of this exploration we want to highlight is the fact that at  each step, we are potentially exploring edges without knowing one of its endpoints. In general, the only way to know  for sure in a random graph $\calH$ which vertices the edge $e$ connect is to find, during exploration, both of its endpoints. However, in the context of our $G_q$, and since the value $f$, on both lifts of $e$, are heavily coupled variables, knowing the endpoints of the edge $e_i$ means exploring at least one lift of both endpoints of $e$. 

\end{remark} 
With this exploration in mind, we want to get an estimate to the quantity $$\Pro_p\Big(e_{k+1}\text{ piv. for } (\calC_0\cap \calM = \emptyset)\Big|(\calC_0\cap \calM = \emptyset)\cap \Expl_k\Big).$$ Henceforth, let us write $(e_{k+1}\text{ piv.})$ for $\big(e_{k+1}\text{ piv. for } (\calC_0\cap \calM = \emptyset)\big)$.

Now, assume we follow a different exploration where we skip revealing the state of $e_{k+1}$ to continue exploring $\calC_o$ until we are forced to explore $e_{k+1}$. We call $\Expl'_k$ the random variable that represents the information contained in $\Expl_j$ where $j$ is the step we are forced to explore $e_{k+1}$ i.e. 
$$\Expl'_k=\Big(\omega,f\mapsto (g(\xx_1),\dots,g(\xx_j)),(\omega_{\ee_1},\dots,\omega_{\ee_{j-1}})\Big)$$ 
where $j$ is such that $\ee'_j=e_{k+1}$ in our new exploration. Let us label $e'_1,\dots e'_j=\ee_1',\dots,\ee'_j$.
Let us also define $\widetilde{\Pro_p}=\Pro^{G_q^n}_p\left(\;\cdot\;\big|\calC_0\cap \calM = \emptyset\right)$ and $\widetilde{\E_p}$ the corresponding expectation operator. 
As $\Expl'_k$ is the result of an exploration that has moved further than $\Expl_k$, it contains more information and because of the inclusion of the corresponding $sigma$-field, we can write $$\widetilde{\Pro_p}(e_{k+1}\text{ piv.} \big|\Expl_k)=\widetilde{\E_p}\left[\Pro^{G_q^n}_p(e_{k+1}\text{ piv.}\big|\Expl_k')\big|\Expl_k\right]$$

We now want to pay close attention to the quantity $\Pro^{G_q^n}_p[e_{k+1}\text{ piv.}\big|\Expl_k']$. Note that when $\Expl_k$ and $(\calC_o\cap\calM = \emptyset)$ hold, the event \big($e_{k+1}$ is pivotal for $(\calC_o\cap\calM = \emptyset)$\big)  implies that $e_{k+1}$ is closed, but one of its endpoint is connected to the origin and the other to a green vertex (such that if $e_{k+1}$ were to be open, the origin would be connected to a green vertex through $e_{k+1}$). We can therefore write $e_{k+1}=\{v_{k+1},w_{k+1}\}$ where:
\begin{itemize}
    \item $v_{k+1}$ is connected to the origin by open edges in $\{e_1,\dots,e_k\}$,
    \item $w_{k+1}$ is connected to a green vertex by open edges that do not belong in $\{e'_1,\dots,e'_{j}\}$,
    \item $e_{k+1}$ is closed,
\end{itemize}
which is a decomposition of the event $\big(e_{k+1}\text{ piv.})$ when $\Expl_k$ and $(\calC_o\cap \calM = \emptyset)$ hold. Note that the first condition is already contained within the event $\Expl_k$. We can also point out that $e_{k+1}$ is closed independently of the first two conditions. By labeling $B_{k+1}=(w_{k+1}$ is connected to a green vertex by open edges that do not belong to $\{e_1,\dots,e_{j}\})$, for all $(\omega,f)$ that satisfy $\Expl_k\cap(\calC_o\cap\calM=\emptyset)$, we can write 

$$\Pro^{G_q^n}_p(e_{k+1}\text{ piv.}\big|\Expl_k')\leq \Pro^{G_q^n}_p(B_{k+1}\big|\Expl_k')$$ 
To conclude our proof, we will use Lemma $\ref{lem:coupl}$ and therefore it suffices to prove that $\Pro^{G_q^n}_p[B_{k+1}\big|\Expl_k']\leq 2m_h(p)$. 
To this end, let us point out that if we assume that $e_{k+1}$ is closed, then $\Expl'_{k}$ prescribes the connected component $\calC_o$. Since $B_{k+1}$ forces us to avoid edges that were part of $\calC_o$ (including $e_{k+1})$, we have to connect $w_{k+1}$ to $\calM$ in the remaining graph $\widetilde{G_q}$ after $\calC_o$. The only issue is that, since $e_{k+1}$ may have an incomplete $\eta$-status, $w_{k+1}$ is random vertex between $y_0$ and $y_1$. Using a union bound, we have
\begin{align*}
    \Pro^{G_q^n}_p(B_{k+1}\big|\Expl_k')&\leq\Pro_p^{\widetilde{G^n_q}}(\widetilde{\calC_{y_0}}\cap\calM\neq\emptyset)+\Pro_p^{\widetilde{G^n_q}}(\widetilde{\calC_{y_1}}\cap\calM\neq\emptyset)\\
    &\leq\Pro_p^{G^n_q}(\calC_{y_0}\cap\calM\neq\emptyset)+\Pro_p^{G^n_q}(\calC_{y_1}\cap\calM\neq\emptyset)\\\\
    &\leq 2m_h(p)
\end{align*}
by Lemma \ref{lem:trou_sto}. Since $\Pro_p^{G_q}(\calC_y\cap\calM\neq\emptyset)= m_h(p)$ we finished our proof. 

\end{proof}
Now that we have highlighted the necessity of Lemma \ref{lem:trou_sto}, we can move on to its proof. Recall that $\calC_o$ and $g_0$ are deterministic in the context of Lemma \ref{lem:trou_sto}.
\begin{proof}[Proof of Lemma \ref{lem:trou_sto}:]
Let us first define the set $\calC^{\star}_o=\widetilde{\pi}^{-1}(\pi(\calC_o))$ which is the set of vertices not in $\calC_o$ but such that their twin vertices were in $\calC_o$. We say an edge in $\widetilde{G_q}$ is type $(k)$ if exactly $k$ of its endpoints are in $\calC^{\star}_o$. This definition is unambiguous since, by definition, we already know which edges are attached to the vertices in $\calC^{\star}_o$. 
We can give some elementary properties of edges given their type. 
\begin{itemize}
    \item Type $(0)$ are the same as they were in $G_q$ because they were not affected by $g_o$ or $\calC_o$ in any way. If an edge is type $(0)$, then its twin edge is present in $\widetilde{G_q}$ and also type $(0)$
    \item Type $(1)$ edges are partially fixed. one of their endpoints has been set by $g_o$, but the other is still random between the two twin vertices. If an edge is type $(1)$, its twin edge is not present in $\widetilde{G_q}$ and was closed in $G_q$. 
    \item Type $(2)$ edges have a fully determined $\eta$-status. Their twin edge are not present in $\widetilde{G_q}$. 
\end{itemize}
We will partition the set $E(\widetilde{G_q})$ accordingly to the edge types as $E_0,E_1$ and $E_2$. If $S$ is a subset of vertices of $V(\widetilde{G_q})$, we define the edge boundary and vertex boundary of $S $ :$$\partial_ES=\big\{e\in E(\widetilde{G_q})|\,f(e)=\{x,y\}\text{ with }x\in S,y\in V(\widetilde{G_q}),y\notin S\big\}$$ $$\partial_VS=\big\{x\notin S\big|\exists e\in E(\widetilde{G_q}), f(e)=\{x,y\},y\in S\big\}$$
We have  $\partial_E \calC^{\star}_o=E_1$ by definition of $E_1$ and $\partial_E$. % A refaire
\begin{figure}[h!]
    \centering
    \includegraphics[scale=0.4]{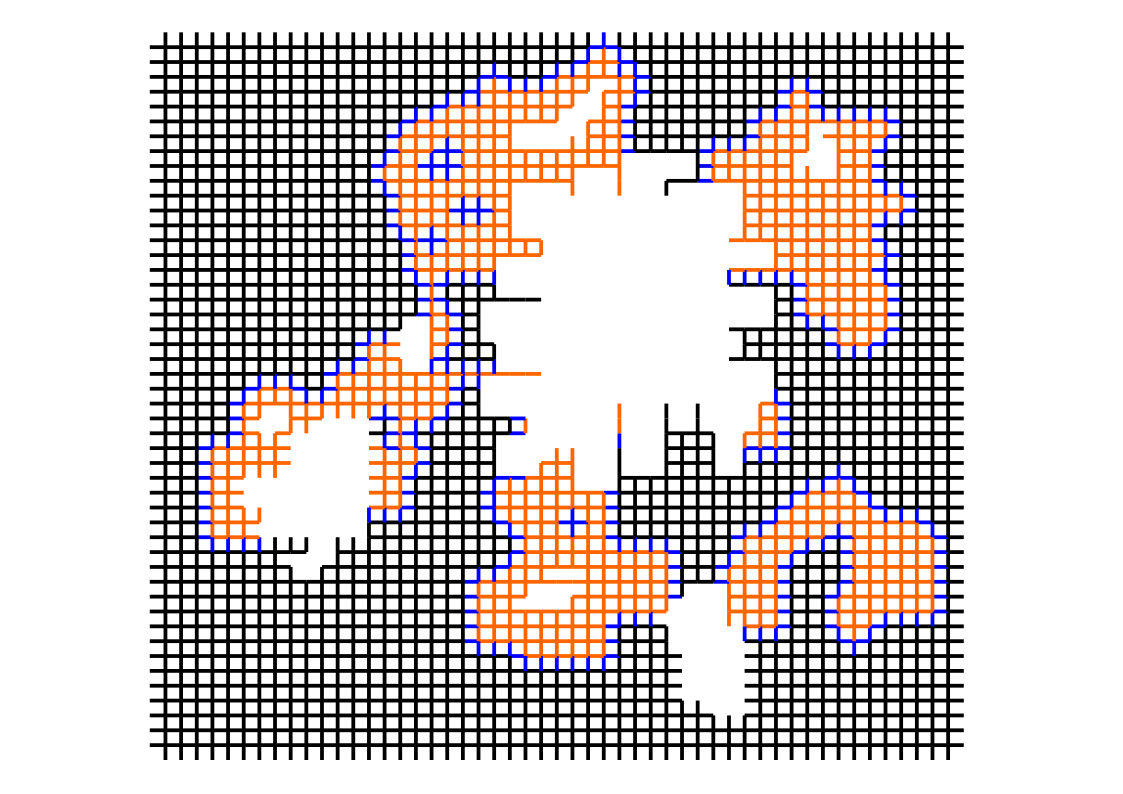}
    \caption{An example of a remaining graph $\widetilde{G_q}$ with $G=\Z^2$ seen from above.}
    \label{fig:Graphe restant}
\justifying
\noindent Since the graph in this figure is seen from above, note that an edge here represent a pair of twin edges in $\widetilde{G_q}$. Holes correspond to pairs of edges which were both explored and fully deleted in $\widetilde{G_q}$. Type $(2)$ edges are represented in orange and type $(1)$ edges in blue. Type $(0)$ edges are drawn in black.
\end{figure}

Let us take a realization $(\omega,f)$ of a percolation in $\widetilde{G_q}$, meaning for the rest of this section we will talk about which edges join which vertices deterministically. We will describe a realization $(\omega^\star,f^\star)$ of a percolation in $G_q$ that couple both variables to give a stochastic domination between the sizes of the cluster of $x$ in $\widetilde{G_q}$ and $G_q$. Let $(U_e)_{e\in V(G)}$ be a set of independent Bernoulli variables with parameter $p$ that will be used to ``fill'' our percolation when needed. 
Let $\calH=(\calC_o^\star\cup\partial_VC_o^\star,E_1\cup E_2)$. This graph is well-defined as type $(2)$ edges have both endpoints in $C^\star_o$ and type $(1)$ edges have one endpoint in $C^\star_o$ and one endpoint in $\partial_VC_o^\star$.
For type $(0)$ edges, we simply set $\omega^\star_e=\omega_e$ and $f^\star(e_i)=f(e_i)$. For the set of type $(2)$ edges, we set $f^\star$ to be defined randomly according to the construction of the distribution of $G_q$. To specify the remaining properties of our coupling, we will need to construct an association between vertices and edges in $\mathcal{H}$ to one of their lifts in $G_q$ (meaning we will designate a distinguished lift for every vertex and edges we come across). We will construct this association by exploring the graph $\mathcal{H}$.
The set $S$ (resp. $A$) will denote the set of vertices (resp. edges) we have already explored in $\widetilde{G_q}$. Our process will continue until $S=C^\star_o$ The function $a:\calC^{\star}_o\cup E_1\cup E_2\to V(G_q)\cup E(G_q)$ 
will be our association, constructed along the way. We fix an arbitrary ordering of the vertices and edge of $\mathcal{H}$.

Since twin edges of type $(1)$ edges and type $(2)$ edges 
are not  in $\widetilde{G_q}$, we may drop their subscript when mentioning them. This is also the case for vertices in $\calC^{\star}_o$. This notation allows us to distinguish vertices and edges in $G_q$ with their counterpart in $\widetilde{G_q}$. 

First, pick the smallest vertex $r$ in $\calC^{\star}_o$
and a corresponding $r_{\delta}$ in $G_q$ where $\delta$ is Bernoulli variable with parameter $1/2$. We associate $r$ to $r_\delta$ and add $r$ to $S$. Then we apply the following iterative process :
\begin{itemize}
    \item We pick the smallest edge $e$ in $E_2$ that is not yet in $A$ and which has at least one endpoint $x$ in $S$ (we choose $x$ as the smallest endpoint if both are in $S$). If such an edge does not exist, we immediately go to the last step. 
    \item Let $e_i$ the lift of $e$ that has $a(x)$ as an endpoint, we call $y$ the other endpoint of $e$. We set $a(e)=e_i$. 
    \item If the edge $e$ is open and if $a(y)$ is not defined, then we set it equal to the other endpoint of $e_i$ and add $y$ to $S$. 
    \item We add $e$ to $A$.
    \item If there is no remaining edge to be picked for the first step, it means we explored all the connected component of the root $r$ in $\calH$. If $S\neq C^\star_o$ we pick a new root $r'$, the smallest vertex in $C^\star_o$ not yet in $S$ and a random lift $ r'_{\delta'}$ and start again at the first step. If $S=C^\star_o$, we are done with process
\end{itemize}
Once we are done, we have constructed the function $a$ and can therefore set $\omega^\star_{a(e)}=\omega_e$. We also set $\omega^\star_{T(a(e))}=U_e$. The only remaining task in our construction is to set the $(\omega,f)$-status of type $(1)$ edges. Let $e$ be a type $(3)$ edge such that $f(e)=\{x,y_j\}$ with $x\in V(E_2)$ and $y_j\notin V(E_2)$. Note that we still specify the height of $y_j$ as outside $V(E_2)$, $y_j$ and $T(y_j)$ are present in $\widetilde{G_q}$. With probability $1/2$, we set $a(e)$ to be either $e_1$ or $e_2$. We then define the $(\omega,f)$-status of $e_i$  by 
$$\left\{
\begin{array}{ll}
       f^\star(a(e))=\{a(x),y_j\}\text{ and }f^\star(T(a(e))=\{T(y_{j}),T(a(x))\}\\
       \omega^\star_{a(e)}=\omega_e
\end{array}
\right.$$
We also set $\omega_{T(e_i)}=U_e$. 
Note that every time a vertex is discovered, it is from a certain edge that points toward this vertex. We then declare this edge to be a ``pioneer edge'' and orient it toward the new vertex. The set of pioneer edges forms an oriented spanning forest of the open connected components of $\calC_o^\star$. For any vertex in $\calC_o^\star$, there is a unique oriented path from a root to this vertex. Note that by construction, $a$ lifts this oriented forest in $G_q$.

\begin{figure}[h!]
    \centering
    \includegraphics[scale=0.30]{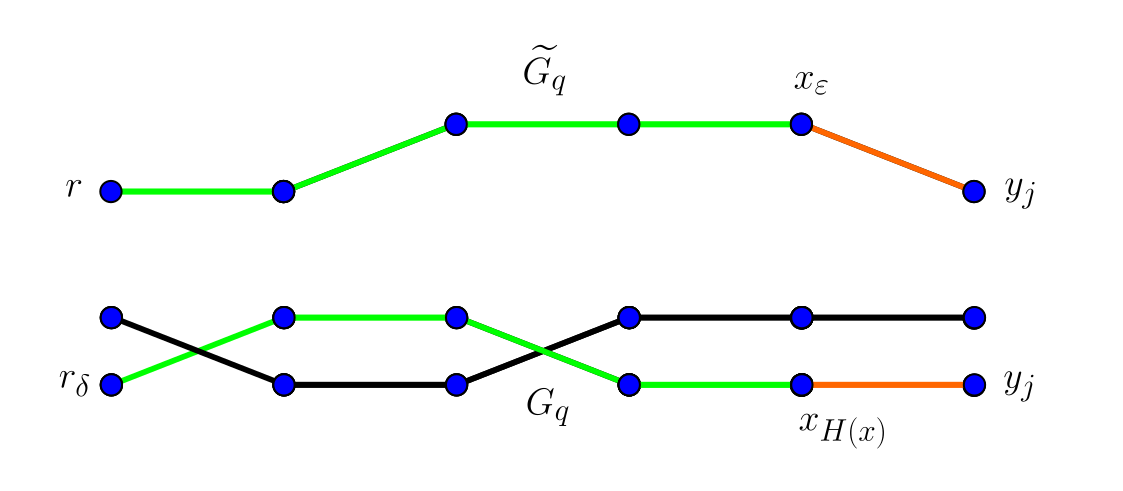}
    \caption{The geometry of a type $(1)$ edge (in green) and its lift depends on the geometry of the paths from the roots $r$ and $r_{\delta}$}. The random variable $H(x)$ describes the height of the vertex associated to $x_{\varepsilon}$.
    \label{fig:Article orientation}
\justifying
\noindent 
\end{figure}

Let us now explain why the construction we described follows the law of $G_q$. We will extend the notion of type of edges in $G_q$ by saying an edge is of type $(k)$ if exactly $k$ of its endpoint are in $\pi^{-1}(\pi(\calC_o))$. This definition is unambiguous, as we can determine the type of edge in $G_q$ without not knowing its ``exact endpoints''.
Note that the type of pairs of edges is still a deterministic notion within the proof of Lemma \ref{lem:trou_sto}
\begin{itemize}
    \item The $(\omega^\star,f^\star)$ status of type $(0)$ edges in $G_q$ is copied from its status in $\widetilde{G_q}$ and therefore has the right joint distribution.
    \item Independently of the status of type $(0)$ edges, the $f^\star$-status of type $(2)$ edges in $G_q$ were set according to the distribution of $G_q$.
    \item The $\omega^\star$-status of type $(0)$ edges was partially set as a lift of the percolation in $\widetilde{G_q}$ and completed with independent Bernoulli variable (similarly to common proofs in percolation, as in \cite{BS96}). The result is therefore  a set of independent Bernoulli variables independent of the state of type $(0)$ edges and the $f^\star$-status of type $(2)$ edges. 
    \item We now look at the crucial case of the $f^\star$-status of type $(1)$ edges. For this explanation, we briefly go back to our description of $G_q$ with the $\eta_e$ variables (we will label $\eta^\star$ the variables that affect the geometry of $G_q$). First, note that this $f^\star$-status is independent of the status of type $(0)$ edges.
    If we write $f(e)=\{x_{\varepsilon},y_{j}\}$ and $f^\star(e_i)=\{y_{j},x_{H(x)}\}$, we can compare the height of the vertices that $e$ and $e_i$ join (see Figure \ref{fig:Article orientation}) to write:$$\eta_e+\varepsilon=\eta^\star_{e}+H(x)$$ in $\Z/2\Z$. The random variable $H(x)$ depends on the variables $\eta^\star_h$ (but independent of $\eta^\star_e$ !) when one follows the oriented path from $r_{\delta}$ to $x_{\varepsilon}$ in $G_q$. Note that all the dependency of $\eta_e^\star$ on the status of type $(2)$ edges is confined to the variable $H(x)$. 
    Because $\eta_e$ is a Bernoulli variable with parameter $1/2$,  $\eta^\star_{e}$ a Bernoulli variable with parameter $1/2$ and independent of $H(x)$. Therefore, $\eta_e^\star$ is independent of the status of type $(2)$ edges. The transition from $\eta^\star_{e}$ to $f^\star$ is done by choosing either $e_1$ or $e_2$ randomly. Also, by this construction, all $\eta_e^\star$ for $e\in E_1$ are independent of each other. 
    \item Finally, the $\omega^\star$-status of type $(1)$ is a Bernoulli percolation with parameter $p$ for the same reason, we have the correct distribution on the $\omega^\star$-status of type $(2)$ edges. It is independent of the status of type $(0)$ and $(2)$ edges, as well as the $f^\star$-status of type $(1)$ edges. 
\end{itemize}
\begin{remark}
    The penultimate step is the reason we are proving sharpness for $q=1/2$ only. In this case only, $\calB(q)$ becomes the Haar measure on $\Z/2\Z$, it is the absorbing element for convolution and allows $\eta_{e}$ to be independent of $H(x)$ (because it is independent of $\eta_e^\star$)
\end{remark}

Finally, we explain the stochastic domination between the size of the connected component of $x$ in $\widetilde{G_q}$ and $G_q$. First, note that any type $(0)$ edge has the same $(\omega,f)$ status in both graphs. For any open type $(2)$ or $(1)$ edges, we can injectivly associate an open edge in $G_q$ through the function $a$. This injection keeps open edges in $\widetilde{G_q}$ open in $G_q$. In particular, this implies it is more probable for $x$ to meet a green vertex in $G_q$ than in $\widetilde{G_q}$ and gives us the desired inequality. \end{proof}

\begin{remark}
The result is not per se a stochastic domination directly between the clusters because while we can associate any open path with type $(2)$ and $(3)$ edges, an open path in $G_q$, the resulting path may have a different geometry as in Figure \ref{fig:Article orientation}. Another reason not captured by Figure \ref{fig:Article orientation} is the fact that $a$ may ``break'' cycles of edges in $\widetilde{G_q}$ while transporting them in $G_q$. 
\end{remark}

The fact that we only proved the exponential decay at $q=1/2$ poses the following question. 
\begin{question}
    Does Theorem \ref{Thm:Sharp} still holds when $q\neq 1/2$ ? 
\end{question}
Since $G_q$ for $q=0$ or $q=1$ is a deterministic transitive graph, the question is already but it remains open for $q\in(0,1/2)\cup(1/2,1)$. 
While we believe the answer is positive, we also believe Lemma \ref{lem:trou_sto} does not hold in certain cases when $q\neq1/2$. Therefore, a new strategy is needed for such a proof.
\paragraph{Acknowledgements}

We are grateful to Paul Rax for pointing out a mistake in an earlier version of the proof of continuity. We would like to give credit to Christoforos Panagiotis for the distinction Quenched/Annealed concerning the exponential decay.
We also want to thank Hugo Vaneuville for the explanation he gave us on his proof of the exponential decay through coupling.  
This work originates from a research internship under the supervision of Sébastien Martineau. We are grateful to him for his help and advice during the internship, as well as during the process of writing of this article.  

\printbibliography

\end{document}